\documentclass[11pt]{amsart}

\usepackage{amsmath}
\usepackage{amsthm}
\usepackage{amsfonts}
\usepackage{xcolor}
\usepackage[english]{babel}
\usepackage[margin=1.3in]{geometry}

\usepackage[latin9]{inputenc}
\usepackage{amsmath}
\usepackage{amsfonts}
\usepackage{amssymb}
\usepackage{amsthm}

\newtheorem{theo}{Theorem}[section]
\newtheorem{prop}[theo]{Proposition}
\newtheorem{coro}[theo]{Corollary}

\theoremstyle{definition}

\theoremstyle{remark}
\newtheorem{rema}[theo]{Remark}

\newcommand{\nwc}{\newcommand}
\nwc{\eps}{\epsilon}
\nwc{\vareps}{\varepsilon}
\nwc{\Oph}{\operatorname{Op}_\hbar}
\nwc{\la}{\langle}
\nwc{\ra}{\rangle}

\nwc{\mf}{\mathbf} 
\nwc{\blds}{\boldsymbol} 
\nwc{\ml}{\mathcal} 

\nwc{\defeq}{\stackrel{\rm{def}}{=}}

\nwc{\cE}{\ml{E}}
\nwc{\cN}{\ml{N}}
\nwc{\cO}{\ml{O}}
\nwc{\cP}{\ml{P}}
\nwc{\cU}{\ml{U}}
\nwc{\cV}{\ml{V}}
\nwc{\cW}{\ml{W}}
\nwc{\tU}{\widetilde{U}}
\nwc{\IN}{\mathbb{N}}
\nwc{\IR}{\mathbb{R}}
\nwc{\IZ}{\mathbb{Z}}
\nwc{\IC}{\mathbb{C}}
\nwc{\IT}{\mathbb{T}}
\nwc{\tP}{\widetilde{P}}
\nwc{\tPi}{\widetilde{\Pi}}
\nwc{\tV}{\widetilde{V}}
\nwc{\supp}{\operatorname{supp}}
\nwc{\rest}{\restriction}

\renewcommand{\phi}{\varphi}

\title[Quantitative equidistribution properties of toral eigenfunctions]
{Quantitative equidistribution properties\\ of toral eigenfunctions}

\author{Hamid Hezari }
\address{Department of Mathematics, UC Irvine, Irvine, CA 92617, USA} \email{hezari@math.uci.edu}

\author[Gabriel Rivi\`ere]{Gabriel Rivi\`ere}
\address{Laboratoire Paul Painlev\'e (U.M.R. CNRS 8524), U.F.R. de Math\'ematiques, Universit\'e Lille 1, 59655 Villeneuve d'Ascq Cedex, France}
\email{gabriel.riviere@math.univ-lille1.fr}

\date{\today}

\begin{document}

\begin{abstract} In this note, we prove quantitative equidistribution properties for orthonormal bases of eigenfunctions of the Laplacian on the rational $d$-torus. We show that the rate of equidistribution of such eigenfunctions is polynomial. We also prove that equidistribution of eigenfunctions holds for symbols supported in balls with a radius shrinking at a polynomial rate.


\end{abstract}

\maketitle
\section{Introduction}

In~\cite{Sh74, Ze87, CdV85}, Shnirelman, Zelditch, and Colin de Verdi\`ere proved that, on a compact connected Riemannian manifold $(M,g)$ without boundary, whose geodesic flow is ergodic for the Liouville measure, the eigenfunctions of the Laplacian are quantum ergodic. Quantum ergodicity means that, for any orthonormal basis of eigenfunctions, there exists a full density subsequence along which the associated microlocal lift on the unit cotangent bundle $S^*M$ tend weakly to the Liouville measure on $S^*M$. The main example of ergodic geodesic flow is given by the geodesic flow on negatively curved manifolds. Thus, in this geometric context, eigenfunctions of the Laplacian are quantum ergodic. On a general compact Riemannian manifold $(M,g)$, the geodesic flow is not ergodic for the Liouville measure, and the above result can be extended by using the ergodic decomposition of the Liouville measure -- see for instance~\cite{Ri13}.

 A natural example is the case of the rational torus $\IT^d=\IR^d/\IZ^d$ endowed with its canonical metric. In this setting, the geodesic flow is not ergodic, but if one considers symbols dependent on $x$ and independent of $\xi$, then an ergodic property holds for such symbols -- see proposition~\ref{p:moment-torus} below. Using this observation, one can show that eigenfunctions become equidistributed\footnote{Note that they do not equidistribute on $S^*\IT^d$.} on the configuration space $\IT^d$~\cite{MaRu12, Ri13}. In this paper, we will give quantitative versions of this equidistribution property of toral eigenfunctions. In fact, these quantitative equidistribution properties are originated and motivated by some conjectures concerning the eigenfunctions of negatively curved manifolds. The two related topics we shall be concerned with are:
\begin{itemize}
\item The rate of equidistribution (Theorem~\ref{p:rate-QE-torus}).  
\item Small scale equidistribution (Corollary~\ref{c:smallball}). 
\end{itemize}
Also, in the last section, as an easy corollary of Zygmund's theorem and the equidistribution results of \cite{MaRu12, Ri13}, we show that the quantum ergodicity holds on the 2-torus for $L^2$ symbols. 

\begin{rema}It is worth mentioning that on the rational torus all the quantum limits of eigenfunctions of the Laplacian are absolutely continuous with respect to the Lebesgue measure ~\cite{Ja97, AM14}. In fact in dimension two, by Zygmund's theorem, there is a uniform bound for the $L^4$ norm of all $L^2$ normalized eigenfunctions, and hence all quantum limits in this case have density functions in $L^2$. This was refined by Jakobson in~\cite{Ja97} where he proved that the density function is a trigonometric polynomial. \end{rema}

Before we state our results, we fix some notations. 

Throughout the paper we denote by $\mathbb{T}^d:=\IR^d/\IZ^d$ the rational torus with the standard metric, and we denote by $dx$ the normalized volume measure induced by the standard metric.  

Let $\alpha>0$ be some fixed positive number. For every $0<\hbar\leq 1$, we consider an orthonormal basis $(\psi^{j}_{\hbar})_{j=1,\ldots, N(\hbar)}$ of the subspace 
$$\ml{H}_{\hbar}:=\mathbf{1}_{[1-\alpha\hbar,1+\alpha\hbar]}(-\hbar^2\Delta)L^2(\mathbb{T}^d),$$ made of eigenfunctions of $-\hbar^2\Delta$. According to the Weyl's law -- see e.g.~\cite{DuGu75}, one has $N(\hbar)\sim  \alpha A_d\hbar^{1-d}$ for some constant $A_d$ depending only on $d$. For each $1\leq j\leq N(\hbar)$, we denote $E_j(\hbar) \in [1-\alpha\hbar,1+\alpha\hbar]$ to be the eigenvalue corresponding to $\psi_j^\hbar$:
$$-\hbar^2\Delta\psi_{\hbar}^j=E_j(\hbar)\psi_{\hbar}^j.$$

\subsection{Rate of equidistribution} Our first result states that
\begin{theo}\label{p:rate-QE-torus} Let $\mathbb{T}^d=\mathbb{R}^d/\IZ^d$ be the rational torus with $d\geq 2$ and let $a$ be an element in $\ml{C}^{\infty}(\IT^d)$ (independent of $\hbar$). Then, there exists some constant $C_a>0$ such that, for any orthonormal basis $(\psi_{\hbar}^j)_{j=1,\ldots N(\hbar)}$ of $\mathbf{1}_{[1-\alpha\hbar,1+\alpha\hbar]}(-\hbar^2\Delta)L^2(\IT^d)$ made of eigenfunctions 
of $-\hbar^2\Delta$,
 $$\frac{1}{N(\hbar)}\sum_{j=1}^{N(\hbar)}\left|\int_{\IT^d}a|\psi_{\hbar}^j|^2dx-\int_{\IT^d}adx\right|^2\leq C_a\hbar^{\frac{2}{3}}.$$
\end{theo}
The fact that this quantity converges to $0$ was already observed in~\cite{MaRu12, Ri13}. The novelty here is that we are able to prove that this convergence holds at a polynomial rate. Although this is a natural question, but to our knowledge it has not been addressed in the literature. This result is a direct consequence of Theorem~\ref{p:QE-torus} below which is slightly more general. In~\cite{MaRu12}, Marklof and Rudnick proved that equidistribution on configuration space also holds for eigenfunctions of a rational polygon, hence it would be natural to understand if one can obtain a polynomial rate of convergence in this setting. In the case of the torus, it would also be natural to understand what the optimal rate of convergence should be depending on the dimension.
\begin{rema}
Let us rewrite this statement using the standard convention (the non-semiclassical notation). We define the eigenfunctions $\psi_j$ to be the nonzero solutions to
$$ -\Delta \psi_j = \lambda_j \psi_j, \quad \lambda_j \geq 0. $$ where the eigenvalues $\lambda_j$ are sorted as
$$0=\lambda_1<\lambda_2 \leq \lambda_3 \dots \to \infty.$$
Suppose $(\psi_j)_{j\in\mathbb{N}}$ is an ONB of $L^2(\mathbb T^d)$ made of eigenfunctions. Then, in this notation, the above result can be written as follows
$$V(a,\lambda):=\frac{1}{N_1(\lambda)}\sum_{j:\lambda-\sqrt{\lambda}\leq\lambda_j\leq\lambda+\sqrt{\lambda}}\left|\int_{\IT^d}a|\psi_j|^2dx-\int_{\IT^d}adx\right|^2=\ml{O}(\lambda^{-\frac{1}{3}}),$$
where (see \cite{DuGu75, DiSj99})
$$N_1(\lambda):=\sharp\left\{j:\lambda-\sqrt{\lambda}\leq\lambda_j\leq\lambda+\sqrt{\lambda}\right\}\sim C_d\lambda^{\frac{d-1}{2}},$$
with $C_d>0$ depending only on $d$. For \emph{chaotic} systems, it is conjectured in the physics literature~\cite{FePe86}, that $V(a,\lambda)$ is of order $\lambda^{\frac{1-d}{2}}$. In the case of negatively curved manifolds, the best known upper bound is $\ml{O}(|\log\lambda|^{-1})$~\cite{Ze94, Sch06}. We emphasize that our general strategy is the same as the ones in these two references: the main inputs are that we are able to use the semiclassical approximation for much longer times and that we have a better control on the error terms due to the exact formulas one has on $\IT^d$. Finally, in the case of Hecke eigenfunctions on the modular surface, we note that the upper bound $\ml{O}(\lambda^{-\frac{1}{2}+\eps})$ was proved in~\cite{LuSa95} for spectral intervals of the form $[\lambda,2\lambda]$.
\end{rema}

\subsection{Small-scale equidistribution} Our next result concerns equidistribution properties of toral eigenfunctions in balls of shrinking radius. This question is motivated by our recent work~\cite{HeRi14}, where we showed that on negatively curved manifolds quantum ergodicity holds for symbols carried on balls whose radius shrink at a logarithmic rate (see also~\cite{Yo13, Han14, LMR15}), and where we found some applications to $L^p$ estimates and the size of nodal sets. In the case of $\mathbb{T}^d$, one can also prove a quantitative equidistribution result where symbols are allowed to depend on $\hbar$. This is the content of the following theorem which is our main result:

\begin{theo}\label{p:QE-torus} Let $\mathbb{T}^d=\mathbb{R}^d/\IZ^d$ be the rational torus with $d\geq 2$. Let $s>\frac{d+4}{2}$, $\nu_0\geq 0$, and $\nu_1\geq 0$. 
Suppose $a=(a_{\hbar})_{0<\hbar\leq 1}\in\ml{C}^{\infty}(\IT^d)$ is a symbol such that 
for every $\beta$ in $\mathbb{N}^d$, there exists $C_{\beta}>0$ satisfying
\begin{equation}\label{e:seminorm}\forall 0<\hbar\leq 1,\ \forall x\in\IT^d,\ |\partial^{\beta}_xa_{\hbar}|\leq C_{\beta}\hbar^{-\nu_1|\beta|}.\end{equation}
Then, there exists $\hbar_0>0$ such that, for any $0<\hbar\leq\hbar_0$, and for any orthonormal basis $(\psi_{\hbar}^j)_{j=1,\ldots N(\hbar)}$ of $\mathbf{1}_{[1-\alpha\hbar,1+\alpha\hbar]}(-\hbar^2\Delta)L^2(\IT^d)$ made of eigenfunctions 
of $-\hbar^2\Delta$, one has
$$\frac{1}{N(\hbar)}\sum_{j=1}^{N(\hbar)}\left|\int_{\IT^d}a_{\hbar}|\psi_{\hbar}^j|^2dx-\int_{\IT^d}a_{\hbar}dx\right|^2\leq C\|a_{\hbar}\|^2_{L^2(\IT^d)}\hbar^{\nu_0}+C\|a_{\hbar}\|^2_{H^{s}(\IT^d)}\hbar^{2-2\nu_0}+C_a\hbar^{2-2(\nu_0+\nu_1)},$$
where $C$ is independent of $a$, $\nu_0$ and $\nu_1$, and where $C_a$ depends only on $\nu_0$, $\nu_1$ and on a finite number\footnote{The $\beta$ involved in the constant depends on the choice of $\nu_0$ and $\nu_1$.} of the constants $C_{\beta}$ appearing in~\eqref{e:seminorm}.
\end{theo}

As was already mentioned, this result implies Theorem~\ref{p:rate-QE-torus} by picking $a$ independent of $\hbar$, $\nu_1=0$ and $\nu_0=\frac{1}{2}$. This theorem also allows us to show that on the rational torus, the eigenfunctions equidistribute on balls whose radius shrink at a polynomial rate. More precisely, by choosing $a_\hbar$ to be certain cutoff functions supported in geodesic balls of radius $\hbar ^{\nu_1}$, and, using an extraction argument (see for instance section 3.2 of \cite{HeRi14} and the proof of corollary~$1.7$ in~\cite{Han14}), we get

\begin{coro}\label{c:smallball} Let $0<\nu_1<\frac{2}{7d+4}$. Then there exists $0<\hbar_0\leq 1/2$ such that given any orthonormal basis $(\psi_{\hbar}^j)_{j=1,\ldots N(\hbar)}$ of $\mathbf{1}_{[1-\alpha\hbar,1+\alpha\hbar]}(-\hbar^2\Delta)L^2(\IT^d)$ made of eigenfunctions 
of $-\hbar^2\Delta$, one can find a full density subsequence $\Lambda_{\nu_1}(\hbar)$ of $\{1,\ldots , N(\hbar)\}$ such that
\begin{equation}\label{e:vol-eps-ball-torus} \forall 0<\hbar\leq\hbar_0, \, \forall x \in \mathbb T^d, \,\forall j\in \Lambda_{\nu_1}(\hbar) : \qquad  a_1\leq\frac{\int_{B(x,\hbar^{\nu_1})}|\psi_{\hbar}^j(x)|^2dx}{\operatorname{Vol}(B(x,\hbar^{\nu_1}))}\leq a_2,
\end{equation}
where the constants $a_1,a_2>0$ are independent of $\hbar$, $x$, and $j$ a,d $B(x,\hbar^{\nu_1})$ denotes the geodesic ball of radius $\hbar^{\nu_1}$ centered at $x$. 
\end{coro}

The fact that $(\Lambda_{\nu_1}(\hbar))_{0<\hbar\leq\hbar_0}$ is of full density exactly means that
$$\lim_{\hbar\rightarrow 0}\frac{|\Lambda_{\nu_1}(\hbar)|}{N(\hbar)}=1.$$
In other words, this statement says that there exists a large proportion of eigenfunctions where the average value of the square of eigenfunctions in shrinking balls of radius $\hbar^{\nu_1}$, are uniformly bounded by two constants. We note that the corollary provides a subsequence of density $1$ subsets that works \emph{uniformly for every point on the torus}. If we had considered only one fixed point $x_0$ in $\IT^d$, we would have obtained a critical exponent of size $\frac{1}{2(d+1)}$ instead of the exponent $\frac{2}{7d+4}$ appearing here. The exponent $\frac{2}{7d+4}$ appearing in this statement is probably not optimal and it is plausible that this exponent can be improved by using for instance methods like the ones used by Bourgain in~\cite{Bo13}. Here, our proof relies only on tools from ergodic theory and semiclassical analysis.

\begin{rema} Using this corollary and the strategy of \cite{HeRi14}, we can in fact improve Sogge's $L^p$ estimates for toral eigenfunctions~\cite{So88}. However the $L^p$ bounds we obtain using this method are not better than the upper bounds proved in~\cite{Zy74, Bo93, Bo13, BoDe14}. 
\end{rema}

\subsection {Quantum ergodicity for $L^2$ observables} 

Motivated by the recent question raised by Zelditch on $L^{\infty}$ quantum ergodicity~\cite{Ze13}, we mention the following nice consequence of the quantum ergodicity property on the $2$-torus:
\begin{coro}\label{p:LinfinityTorus} Let $\mathbb{T}^2=\mathbb{R}^2/\IZ^2$ be the rational $2$-torus. Then, for any orthonormal basis $(\psi_{\hbar}^j)_{j=1,\ldots N(\hbar)}$ of $\mathbf{1}_{[1-\alpha\hbar,1+\alpha\hbar]}(-\hbar^2\Delta)L^2(\IT^2)$ made of eigenfunctions 
of $-\hbar^2\Delta$, there exists a full density subsequence $\Lambda (\hbar)$ of $\{1, 2, \dots, N(\hbar)\}$ such that, for all $a(x)\in L^2(\IT^2)$ (independent of $\hbar$),
\begin{equation} \label{QE2} \lim_{\hbar \to 0, j \in \Lambda(\hbar)} \int_{\IT^2}a(x)|\psi_{\hbar}^j(x)|^2dx= \int_{\IT^2} a(x) dx. \end{equation}

\end{coro}
The important point in this statement is that convergence holds for any observables in $L^2(\IT^2)$, and not only $\ml{C}^0$ ones (as it is usually the case in quantum ergodicity statements). In particular, it holds for the characteristic function of \emph{any} measurable subset of $\IT^2$. In fact, Zelditch conjectured in~\cite{Ze13} that~\eqref{QE2} holds for any $a$ in $L^{\infty}(M)$ provided $(M,g)$ is a \emph{negatively curved manifold}. As will be explained in section~\ref{s:coro}, this corollary follows directly from a classical result of Zygmund~\cite{Zy74} combined to the above quantum ergodicity property. We emphasize that we do not need all the strength of the above theorems and that this result could be in fact deduced directly from the results in~\cite{MaRu12, Ri13}.

We will now give the proof of Theorem~\ref{p:QE-torus} from which all the other results follow. As in the case of negatively curved manifolds, we will first prove a result on the rate of convergence of Birkhoff averages. Then, we will implement 
this result in the classical proof of quantum ergodicity, and we will have to optimize the size of the different remainders to get our results.

\begin{rema} After communicating this note to Zeev Rudnick, he informed us that part of the above results can in fact be improved using methods of more arithmetic nature. The proof presented here only makes use of standard tools of Fourier analysis, and it is modeled on arguments similar to the ones used to prove rate of quantum ergodicity on negatively curved manifolds~\cite{Ze94, Sch06}. 
\end{rema}

\section{Convergence of Birkhoff averages}

We start with the following proposition which gives us the rate of equidistribution for observables depending only on the $x$ variable:
\begin{prop}\label{p:moment-torus} Let $\IT^d=\IR^d/\IZ^d$ with $d\geq 2$. There exists $C_d>0$ such that, for every $a$ in $\ml{C}^{\infty}(\IT^d)$, one has
$$\int_{S^*\IT^d}\left|\frac{1}{T}\int_0^Ta(x+t\xi)dt-\int_{\IT^d}a(x)dx\right|^{2}dxd\xi\leq \frac{C_d\|a\|_{{L}^{2}(\IT^d)}^{2}}{T}.$$
\end{prop}

\begin{proof} Let $a$ be a smooth function on $\IT^d$. We write its Fourier decomposition $a:=\sum_{k\in\mathbb{Z}^d}a_k e_k$, where $e_k(x):=e^{2i\pi k.x}.$ We set
$$V(a,T):=\int_{\IT^d}\int_{\mathbb{S}^{d-1}}\left|\frac{1}{T}\int_0^T a(x+t\xi)dt-\int_{\IT^d}a(x)dx\right|^{2}d\xi dx.$$
First, we perform integration in the $x$ variable and we find that
$$V(a,T)=\frac{1}{T^2}\sum_{k\in\IZ^d-\{0\}}|a_k|^2\int_{\mathbb{S}^{d-1}}\left|\int_0^Te^{2i\pi k.t\xi}dt\right|^2d\xi.$$
Now we would like to estimate the integral in each term of the above sum. Using the spherical symmetry first, and then calculating the $dt$ integral, we get
\begin{align*} \int_{\mathbb{S}^{d-1}}\left|\int_0^Te^{2i\pi k.t\xi}dt\right|^2d\xi &=\int_{\mathbb{S}^{d-1}}\left|\int_0^Te^{2i\pi \|k \| t\xi_1}dt\right|^2d\xi 
\\ &= \int_{\mathbb{S}^{d-1}} \frac{\sin^2{(\pi T \|k \| \xi_1)} }{\pi^2 || k||^2 \xi_1^2}d\xi \end{align*}
By using the spherical coordinates and putting $ \xi_1 =\cos \phi$ , $0 \leq \phi \leq \pi$, we obtain
\begin{align*} 
 \int_{\mathbb{S}^{d-1}} \frac{\sin^2{(\pi T\|k \| \xi_1)} }{\pi^2 || k||^2 \xi_1^2}d\xi 
&= C \int_{0}^\pi \frac{\sin^2{(\pi T \|k \| \cos \phi)} }{\pi^2 || k||^2 \cos^2 \phi} (\sin \phi)^{d-2}\, d\phi,
\end{align*}
where the constant $C$ is the value of the integral with respect to the remaining spherical variables. The change of variable $s=\cos \phi$, turns this last integral into 
$$2\int_{0}^1 \frac{\sin^2{(\pi T \|k \| s)} }{\pi^2 || k||^2 s^2} (\sqrt{1-s^2})^{d-3}\, ds. $$ 
We then split this integral into integrals over $[0, \delta]$ and its complement $[\delta, 1]$, where $0< \delta <1$. Clearly  $$\int_{\delta}^1 \frac{\sin^2{(\pi T \|k \| s)} }{\pi^2 || k||^2 s^2} (\sqrt{1-s^2})^{d-3}\, ds \leq \frac{1}{\pi^2||k||^2 \delta^2}.$$ To estimate the integral on $[0, \delta]$, we use the substitution $u= \pi T ||k|| s$. Hence, since $d \geq 2$ we get
\begin{align*}\int_{0}^\delta \frac{\sin^2{(\pi T \|k \| s)} }{\pi^2 || k||^2 s^2} (\sqrt{1-s^2})^{d-3}\, ds 
& \leq \frac{T}{\pi ||k||\sqrt{1-\delta^2}} \int_{0}^{\pi T ||k|| \delta} \frac{\sin^2{u} }{u^2}\, du
\\& \leq \frac{T}{\pi ||k||\sqrt{1-\delta^2}} \int_{0}^{\infty} \frac{\sin^2{u} }{u^2}\, du
\\&=\frac{T}{2||k||\sqrt{1-\delta^2}}.\end{align*}
Therefore, by choosing $\delta =\frac{1}{\sqrt{||k||T}}$ and $T\geq 2$, we get
$$ V(a, T) \leq \frac{C'}{T} \sum_{ k\neq 0} \frac{|a_k|^2}{\|k \|} \leq \frac{C'\|a\|_{{L}^{2}(\IT^d)}^{2}}{T}. $$ 
for some uniform constant $C'>0$.
\end{proof}

\section{Proof of Theorem~\ref{p:QE-torus}}\label{ss:rate-torus}

We fix $a$ in $\ml{C}^{\infty}(\IT^d)$ that potentially depends on $\hbar$, even if we omit the index $\hbar$ in order to alleviate the notations. We also suppose that $a$ belongs to a nice class of symbols. More precisely, there exists $ \nu_1\geq 0$ such that, for every $\alpha$ in $\mathbb{N}^d$, one can find $C_{\alpha}>0$ such that
\begin{equation}\label{e:seminorms}\forall x\in\IT^d,\ |\partial^{\alpha}_xa|\leq C_{\alpha}\hbar^{-\nu_1|\alpha|}.\end{equation}
Without loss of generality, we will also suppose that $a$ is real valued. We set $\overline{a}:=a-\int_{\IT^d}adx$. Our goal is to give an upper bound on the following quantity:
$$V_{\hbar,2}(a)=\frac{1}{N(\hbar)}\sum_{j=1}^{N(\hbar)}\left|\int_{\IT^d}a(x)|\psi_{\hbar}^j(x)|^2dx-\int_{\IT^d}adx\right|^{2}.$$

\subsection{Applying Egorov's theorem} \label{ss:egorov}

We rewrite the previous expression as follows
$$V_{\hbar,2}(a)=\frac{1}{N(\hbar)}\sum_{j=1}^{N(\hbar)}\left|\left\la\psi_{\hbar}^j,\overline{a}\psi_{\hbar}^j\right\ra_{L^2}\right|^{2}.$$
One of the main differences with the negatively curved case treated in~\cite{Ze94, Sch06, Han14, HeRi14} is that we can consider much longer semiclassical times using the fact that we are on $\IT^d$.  Precisely, we fix $\nu_0>0$ and $T=T(\hbar):=\hbar^{-\nu_0}$. We introduce the averaged operator
$$A(T,\hbar):=\frac{1}{T}\int_0^Te^{-\frac{it\hbar\Delta}{2}}\overline{a}e^{\frac{it\hbar\Delta}{2}}dt.$$
Using the fact that $\psi_{\hbar}^j$ is an eigenmode for every $1\leq j\leq N(\hbar)$, one can write that
$$V_{\hbar,2}(a)=\frac{1}{N(\hbar)}\sum_{j=1}^{N(\hbar)}\left|\left\la\psi_{\hbar}^j,A(T,\hbar)\psi_{\hbar}^j\right\ra_{L^2}\right|^{2}.$$
Using the Cauchy-Schwarz inequality, we find that
$$V_{\hbar,2}(a)\leq\frac{1}{N(\hbar)}\sum_{j=1}^{N(\hbar)}\left\la\psi_{\hbar}^j,A(T,\hbar)^2\psi_{\hbar}^j\right\ra_{L^2}.$$
Recall now that $(\psi_{\hbar}^j)_{j=1}^{N(\hbar)}$ is an orthonormal basis of the space 
$$\ml{H}_{\hbar}=\mathbf{1}_{[1-\alpha\hbar,1+\alpha\hbar]}(-\hbar^2\Delta)L^2(\IT^d).$$
We will now take advantage of the fact that we are on the torus. Precisely, we have 
\begin{equation}\label{e:step0}V_{\hbar,2}(a)\leq\frac{1}{N(\hbar)}\sum_{k\in\IZ^d:(2\pi\hbar\|k\|)^2\in[1-\alpha\hbar,1+\alpha\hbar]}\left\| A(T,\hbar)e_k\right\|_{L^2}^2.\end{equation}
where $e_k(x):=e^{2i\pi k.x}.$ We write the following exact formula
$$\left(A(T,\hbar)e_k\right)(x)=\left(\frac{1}{T(\hbar)}\int_0^{T(\hbar)}\sum_{p\in\IZ^d-\{0\}}\hat{a}_p e^{\frac{it\hbar(2\pi)^2\|p\|^2}{2}} e_p(x+2\pi\hbar t k ) dt \right)e_k(x),$$
where $a:=\sum_{p\in\IZ^d}\hat{a}_pe_p(x).$ We now fix $s>\frac{d+4}{2}.$ We note that, uniformly for $x$ in $\IT^d$ and $t$ in $\IR$, one has
$$\left|\sum_{p\in\IZ^d-\{0\}}\hat{a}_p e^{\frac{it\hbar(2\pi)^2\|p\|^2}{2}} e_p(x+2\pi\hbar t k )-\overline{a}(x+2\pi k\hbar t)\right|\leq2\pi^2\hbar t\sum_{p\neq 0}\|p\|^2|\hat{a}_p|\leq c_s \hbar t\|\overline{a}\|_{H^{s}},$$
for some constant $c_s>0$ depending only on $d$ and $s$. Implementing this in our upper bound~\eqref{e:step0}, we obtain that
\begin{equation}\label{e:step1}V_{\hbar,2}(a)\leq\frac{2}{N(\hbar)}\sum_{k\in\IZ^d:(2\pi\hbar\|k\|)^2\in[1-\alpha\hbar,1+\alpha\hbar]}\int_{\IT^d}\left(\frac{1}{T(\hbar)}\int_0^{T(\hbar)}\overline{a}(x+2\pi k\hbar t)dt\right)^2dx+\|\overline{a}\|_{H^{s}}^2\ml{O}(\hbar^{2-2\nu_0}),\end{equation}
where the constant in the remainder is independent of $a$.

\subsection{Trace asymptotics} In order to compute the previous expression, we proceed as in~\cite{DuGu75} -- see~\cite{DiSj99} (Ch.~$11$) for a semiclassical version. We will in fact follow the presentation of Prop.~$1$ in~\cite{Sch06} and we will take advantage of the fact that we are working on $\IT^d$. 

Regarding~\eqref{e:step1}, we now have to estimate
$$\tilde{V}_{\hbar,2}(a):=\frac{1}{N(\hbar)}\sum_{k\in\IZ^d:(2\pi\hbar\|k\|)^2\in[1-\alpha\hbar,1+\alpha\hbar]}\int_{\IT^d}b_{\hbar}(x,2\pi k\hbar)dx.$$
where we set
$$b_{\hbar}(x,\xi):=\chi_1(\|\xi\|^2)\left(\frac{1}{\hbar^{-\nu_0}}\int_0^{\hbar^{-\nu_0}}a(x+t\xi)dt\right)^2,$$
with $0\leq\chi_1\leq 1$ a smooth cutoff function which is equal to $1$ in a small neighborhood of $1$ and is $0$ outside a slightly bigger neighborhood, say outside $[1/4,4]$. We fix a smooth function $ \rho\geq 0$ in the Schwartz class $\ml{S}(\IR)$, which is $\geq 1$ on the interval $[-\alpha,\alpha]$. We also suppose that $\hat{\rho}$ has compact support, say that the support is included in $[-1/8,1/8]$.

\begin{rema} In order to construct such a function, one can start from a nonzero smooth even function $f\geq 0$ which is compactly supported in $[-1/16,1/16]$ and take $\rho$ to be the inverse Fourier transform of $A_0 f*f$ with $A_0>0$ large enough. We note that the function $\hat{\rho}$ satisfies $\hat{\rho}'(0)=0$.
 
\end{rema}
We can then write that
$$\tilde{V}_{\hbar,2}(a)=\frac{1}{N(\hbar)}\sum_{k\in\IZ^d}\rho\left(\frac{4\pi^2\|k\|^2\hbar^2-1}{\hbar}\right)\int_{\IT^d}b_{\hbar}(x,2\pi k\hbar)dx,$$
and thus
$$\tilde{V}_{\hbar,2}(a)=\frac{1}{N(\hbar)}\sum_{k\in\IZ^d}\int_{\IR}\hat{\rho}(\tau)e^{-\frac{i\tau}{\hbar}}e^{\frac{i4\pi^2\tau\|k\|^2\hbar^2}{\hbar}}\int_{\IT^d}b_{\hbar}(x,2\pi k\hbar)dxd\tau.$$
Thanks to the Poisson summation formula, we get
\begin{equation}\label{e:poisson}\tilde{V}_{\hbar,2}(a)=\frac{1}{N(\hbar)}\int_{\IT^d}\left(\sum_{l\in\IZ^d}\frac{1}{(2\pi\hbar)^d}\int_{\IR\times\IR^d}\hat{\rho}(\tau)e^{i\frac{\tau(\|\xi\|^2-1)-\xi.l}{\hbar}}b_{\hbar}(x,\xi)d\tau d\xi\right)dx.\end{equation}
We will now make use of the stationary (and non-stationary) phase lemma. To do so, we fix $l$ in $\IZ^d$ and we denote by $\varphi_l(\tau,\xi)$ to be the phase function of the above oscillatory integral. We observe that, for $l\neq 0$, one has $\|d_{\xi}\varphi_l\|\geq\|l\|-2\tau\|\xi\|\geq \|l\|-1/2$, for $\tau$ in the support of $\hat{\rho}$ and $\|\xi\|^2$ in the support of $\chi_1$. For $l\neq 0$, we introduce the operator
$$P_l:=\frac{\hbar}{i}\frac{d_{\xi}\varphi_l.d_{\xi}}{\|d_{\xi}\varphi_l\|^2}.$$
We perform $N$ integration by parts using this operator and we find that, for every $x$ in $\IT^d$,
$$\left|\frac{1}{(2\pi\hbar)^d}\int_{\IR\times\IR^d}\hat{\rho}(\tau)e^{i\frac{\tau(\|\xi\|^2-1)-\xi.l}{\hbar}}b_{\hbar}(x,\xi)d\tau d\xi\right|\leq C\frac{\hbar^{N(1-\nu_0-\nu_1)-d}}{(\|l\|-1/2)^N},$$
for some uniform constant $C>0$ that depends only on $\rho$, $a$ and $d$. Using the upper bound~\eqref{e:poisson} and taking $N$ large enough in the previous equation, we get
\begin{equation}\label{e:upperbound}V_{\hbar,2}(a)\leq\frac{1}{N(\hbar)}\int_{\IT^d}\left(\frac{1}{(2\pi\hbar)^d}\int_{\IR\times\IR^d}\hat{\rho}(\tau)e^{i\frac{\tau(\|\xi\|^2-1)}{\hbar}}b_{\hbar}(x,\xi)d\tau d\xi\right)dx+\ml{O}(\hbar^{2(1-\nu_0-\nu_1)}).\end{equation}
We now disintegrate the measure $d \xi$ along the energy layers $\{\|\xi\|^2-1=E\}$ (or in other words we use the coarea formula), to write for every $x$ in $\IT^d$,
$$\int_{\IR\times\IR^d}\hat{\rho}(\tau)e^{i\frac{\tau(\|\xi\|^2-1)}{\hbar}}b_{\hbar}(x,\xi)d\tau d\xi=\int_{\IR}\int_{-1}^{+\infty}\hat{\rho}(\tau)e^{i\frac{\tau E}{\hbar}}\la b_{\hbar}\ra(x,E)dEd\tau,$$
where
$$\la b_{\hbar}\ra(x,E):=\int_{\|\xi\|^2-1=E}b_{\hbar}(x,\xi)dL_E(\xi).$$
We can now use the stationary phase formula and the fact that $\hat{\rho}'(0)=0$ -- see for instance~\cite{Zw12} (Ch.~$3$). Precisely, we find that
\begin{equation}\label{e:stat-phase}\int_{\IR\times\IR^d}\hat{\rho}(\tau)e^{i\frac{\tau(\|\xi\|^2-1)}{\hbar}}\tilde{b}_{\hbar}(x,\xi)d\tau d\xi=\hat{\rho}(0)2\pi\hbar \left (\int_{\|\xi\|^2=1}b_{\hbar}(x,\xi)dL_0(\xi)+\ml{O}(\hbar^{2(1-\nu_0-\nu_1)}) \right).\end{equation}
As before we take $N$ large enough (depending only on $\nu_0$, $\nu_1$ and $d$) in the stationary phase lemma to ensure that the remainder is of order $\ml{O}(\hbar^{2(1-\nu_0-\nu_1)})$. We give emphasis that the remainder term is of the form $\hbar^{2(1-\nu_0-\nu_1)}$ and not $\hbar^{1-2(\nu_0+\nu_1)}$: this is due to the fact that $\hat{\rho}'(0)=0$ and to the symmetry of the phase function -- e.g. Theorem~$3.17$ of~\cite{Zw12}. 


Now by the Weyl's law, we know that $N(\hbar)\sim \alpha C_d\hbar^{1-d}$, for some constant depending only on $d$ -- see for example Ch.~$11$ of~\cite{DiSj99}. Thus, combining~\eqref{e:step1},~\eqref{e:upperbound} and~\eqref{e:stat-phase}, we have that
$$V_{\hbar,2}(a)\leq C_0 \int_{\IT^d}\int_{\mathbb{S}^{d-1}}\left|\frac{1}{\hbar^{-\nu_0}}\int_0^{\hbar^{-\nu_0}}a(x+t\xi)dt\right|^2d\xi dx+C_s'\|\overline{a}\|^2_{H^{s}}\hbar^{2-2\nu_0}+\ml{O}(\hbar^{2(1-\nu_0-\nu_1)}),$$
where $C_0$ depends only on $d$ and on the choice of $\rho$, and where $C_s'$ depends only on $s$. We also note that the constant in the remainder depends only on finitely many of the $C_{\alpha}$ appearing in~\eqref{e:seminorms}.

\subsection{The conclusion}

We can now apply Proposition~\ref{p:moment-torus} and we finally find that
$$V_{\hbar,2}(a)\leq C_0C_d\|a\|_{{L}^{2}(\IT^d)}^{2}\hbar^{\nu_0}+\ml{O}(\hbar^{2(1-\nu_0-\nu_1)})+C_0'\|\overline{a}\|^2_{H^{s}(\IT^d)}\hbar^{2-2\nu_0},$$
which concludes the proof of Theorem~\ref{p:QE-torus}.

\section{Proof of Corollary \ref{p:LinfinityTorus}}
\label{s:coro}

The proof of this result is a direct consequence of the quantum ergodicity property on the $2$-torus, of Zygmund's theorem on the $L^4$ norms of the eigenfunctions on the 2-torus~\cite{Zy74}, and of the Banach-Alaoglu theorem. 

Let $(\psi_{\hbar}^j)_{j=1,\ldots N(\hbar)}$ be an orthonormal basis of $\mathbf{1}_{[1-\alpha\hbar,1+\alpha\hbar]}(-\hbar^2\Delta)L^2(\IT^2)$ made of eigenfunctions 
of $-\hbar^2\Delta$ on the rational 2-torus. Then by Theorem~\ref{p:rate-QE-torus}, there exists a full density subsequence of $\Lambda (\hbar)$ of $\{1, 2, \dots, N(\hbar)\}$ such that for all $a(x)\in C^0(\IT^2)$
\begin{equation} \label{QE3} \lim_{\hbar \to 0, j \in \Lambda(\hbar)} \int_{\IT^2}a(x)|\psi_{\hbar}^j(x)|^2dx= \int_{\IT^2} a(x) dx. \end{equation} 
We refer to section~$15.4$ in~\cite{Zw12} for the details of the extraction argument. We want to show that~\eqref{QE3} holds for all $a \in L^2(\IT^2)$. To prove this, we first note that by Zygmund's theorem \cite{Zy74}, there exists a uniform  constant $A$ such that 
$$ \int_{\IT^2} |\psi_{\hbar}^j(x)|^4 dx \leq A^4.$$ 
Thus, the sequence $\ml{F}:=(|\psi_{\hbar}^j|^2)_{1\leq j\leq N(\hbar), 0<\hbar\leq 1}$ is bounded in $L^2$. By the Banach-Alaoglu theorem, it is relatively compact for the weak-$\star$ topology on $L^2(\IT^2)$. On the other hand, by \eqref{QE3}, $\mathcal F$ has at most one weak-$\star$ limit in $L^2$, and that is the constant function $1$. This proves the corollary.



\section*{Acknowledgments}

The first author is partially supported by the NSF grant DMS-0969745. The second author is partially supported by the Agence Nationale de la Recherche through the Labex CEMPI (ANR-11-LABX-0007-01) and the ANR project GeRaSic (ANR-13-BS01-0007-01). We thank  Fabricio Maci\` a and Henrik Uebersch\"ar for discussions related to the results appearing in this note. Finally, we warmly thank Zeev Rudnick for his careful reading of a preliminary version of this note and for pointing us an improvement in the argument of paragraph~\ref{ss:egorov}.

\end{document}